\documentclass[12pt]{article}

\usepackage[reqno]{amsmath}
\usepackage{amssymb, amsthm}
\usepackage{enumerate,arydshln} 
\usepackage{mathrsfs}
\usepackage{graphicx}    
\usepackage{tikz}
\newtheoremstyle{plainsl}%
	{\topsep}
	{\topsep}
	{\slshape} 
	{}
	{\normalfont\bfseries}
	{.}
	{ }
	{}

\swapnumbers

\theoremstyle{plainsl}
\newtheorem{theorem}{Theorem}[section]
\newtheorem{lemma}[theorem]{Lemma}

\newtheorem{prop}[theorem]{Proposition}

\newcommand\sqr[2]{{\vbox{\hrule height.#2pt
    \hbox{\vrule width.#2pt height#1pt \kern#1pt
        \vrule width.#2pt}\hrule height.#2pt}}}

\renewcommand\qed{%
	\ifmmode\eqno\sqr53
	\else\nolinebreak\ \hfill\sqr53\medbreak\fi}

\numberwithin{equation}{section}

\newcommand{\erd}{Erd\H{o}s}
\newcommand{\multichoose}[2]{\left(\! \binom{#1}{#2}\!\right)}
\newcommand\fim{{\mathcal G}} 
\newcommand\fis{{\mathcal F}} 
\newcommand\sfm[1]{\mathcal G_{(#1)}} 
\newcommand\pr{\,^\prime}
\newcommand \sfs[1]{\mathcal F_{(#1)}} 
\title{Intersection theorems for multisets}
\author{Karen Meagher\\ 
\small  Department of Mathematics and Statistics \\[-0.8ex]
\small University of Regina,  Regina, Saskatchewan, Canada\\[-0.8ex]
\small \texttt{karen.meagher@uregina.ca}\\
Alison Purdy\\
\small  Department of Mathematics and Statistics \\[-0.8ex]
\small University of Regina,  Regina, Saskatchewan, Canada\\[-0.8ex]
\small \texttt{purdyali@uregina.ca}
}

\begin{document}
\maketitle

\abstract{Let $k$, $t$ and $m$ be positive integers.  A $k$-multiset of $[m]$ is a collection of $k$ integers from the set $\{1,\dots, m\}$ in which the integers can appear more than once.  We use graph homomorphisms and existing theorems for intersecting and $t$-intersecting $k$-set systems to prove new results for intersecting and $t$-intersecting families of $k$-multisets.  These results include a multiset version of the Hilton-Milner theorem and a theorem giving the size and structure of the largest $t$-intersecting family of $k$-multisets of an $m$-set when $m \geq 2k-t$.}
\section{Introduction}\label{intro}

In this paper, we show that the method used in~\cite{MR2861399} to prove a natural extension of the famous \erd-Ko-Rado theorem to multisets can be used to prove additional intersection theorems for multisets.

We prove a multiset version of the Hilton-Milner theorem; this result gives the largest family of intersecting multisets satisfying the condition that the intersection of all of the multisets in the family is empty. We determine the largest family of $k$-multisets having the property that no more than $s$ of the multisets from the family can be pairwise disjoint. A related question that we answer is, what is the largest family of multisets that can be partitioned into two intersecting families? We also consider multisets that have intersection of size at least $t$, so are $t$-intersecting. We prove a theorem giving the size and structure of the largest $t$-intersecting family of $k$-multisets of an $m$-set when $m \geq 2k - t$.  Finally, we prove a version of the Hilton-Milner theorem for $t$-intersecting multisets.

In Section~\ref{intro}, we introduce notation and provide some background information on the \erd-Ko-Rado theorem and the known results for intersecting families of multisets.   Additional known results for intersecting set systems are stated in Section~\ref{int sets}.  In Section~\ref{int}, we extend the results from Section~\ref{int sets} to families of multisets.  Section~\ref{t-int} contains our  results concerning $t$-intersecting families of multisets.  In Section 5, we discuss some open problem for multisets.

\subsection{Notation and definitions}

Throughout this paper, small letters are used to denote integers, capital letters are used for sets (and multisets) of integers, and script capital letters are used for collections of objects. The set of integers from $x$ to $y$ inclusive is represented by $[x,y]$.  If $x=1$, this is simplified to $[y]$. 

 A $k$-set (or $k$-subset) is a set of cardinality $k$ and the collection of all $k$-subsets of $[n]$ is denoted by $\binom{[n]}{k}$.  We say that a collection of sets is \textsl{intersecting} if every pair of sets in the collection is intersecting and that it is $t$-intersecting if every pair has at least $t$ elements in common.  Collections are said to be \textsl{isomorphic} if one can be obtained from the other by a permutation of the underlying set.

\subsection{Background}

The \erd-Ko-Rado theorem is an important result in extremal set theory that gives the size and structure of the largest intersecting $k$-subset system from $[n]$.  It appeared in a paper published in 1961~\cite{MR0140419} 
 which contains two main theorems.  The first of these is commonly stated as follows:  
\begin{theorem}[\erd, Ko and Rado~\cite{MR0140419}]
Let $k$ and $n$ be positive integers with $n \geq 2k$.  If $\mathcal F$ is a collection of intersecting $k$-subsets of $[n]$, then \[|\mathcal F| \leq {n-1 \choose k-1}.\]  Moreover, if $n >2k$, equality holds if and only if $\mathcal F$ is a collection of all the $k$-subsets from $[n]$ that contain a fixed element from $[n]$.
\end{theorem}

A Kneser graph, denoted by $K(n,k)$, is a graph whose vertices are the $k$-subsets of $[n]$.  Two vertices are adjacent if and only if the corresponding $k$-subsets are disjoint.  Thus an independent set of vertices is an intersecting set system and the cardinality of the largest independent set is equal to $\binom{n-1}{k-1}$.  (It is assumed that $n \geq 2k$ since otherwise the graph would be the empty graph.)  

A second theorem in~\cite{MR0140419} gives the size and structure of the largest $t$-intersecting $k$-subset system provided that $n$ is sufficiently large relative to $k$ and $t$.  A later theorem due to Ahlswede and Khachatrian extends this result to all values of $n$, $k$ and $t$.  In virtually all cases, the families attaining maximum size are isomorphic to $\sfs{r}$, where $\sfs{r}$ is defined as follows.
For $t,k,n \in \mathbb N$, let
\[ \sfs{r}= \{A \in \binom{[n]}{k}: \left| A  \cap [t+2r]\right| \geq t+r\}.\] 

In the following statement of the theorem, it is assumed that $n >2k-t$ since for $n \leq 2k-t$, the collection of all $k$-subsets of $[n]$ is $t$-intersecting.

\begin{theorem}[Ahlswede and Khachatrian~\cite{MR1429238}]\label{AK}
Let $t$, $k$ and $n$ be positive integers with $t \leq k\leq n$ and let $r$ be a non-negative integer such that $r \leq k-t$.  \begin{enumerate} \item If \[(k-t+1)\left(2 + \frac{t-1}{r+1}\right) <n < (k-t+1)\left(2 + \frac{t-1}{r}\right),\]
then $\sfs{r}$ is the unique (up to isomorphism) $t$-intersecting $k$-set system with maximum size.  (By convention, $\frac{t-1}{r} = \infty$ for $r=0$.)\\

\item If $n=(k-t+1)(2 + \frac{t-1}{r+1})$, then $\left| \sfs{r}\right|= \left| \sfs{r+1} \right|$ and a system of maximum size will equal (up to isomorphism) $\sfs{r}$ or $\sfs{r+1}$.  
\end{enumerate}
\end{theorem}

In this paper, we present intersection theorems for multisets.  A multiset is a generalization of a set in which an element may appear more than once.  We define the cardinality of a multiset as the total number of elements in the multiset including repetitions.  A $k$-multiset is a multiset of cardinality $k$ and the collection of all $k$-multisets containing elements from $[m]$ is denoted by $\multichoose{[m]}{k}$.  The size of this collection is denoted by $\multichoose{m}{k}$ and is equal to $\binom{m+k-1}{k}$.

Let $\mathrm{m}(i,A)$  denote the multiplicity of the element $i$ in the multiset $A$.  We define the intersection of two multisets, $A$ and $B$, as the multiset $C$ where $\mathrm{m}(i,C)=\min\{\mathrm{m}(i,A), \mathrm{m}(i,B)\}$ for each $i \in [m]$.  This definition can be applied to the intersection of a multiset and a set by considering the set to be a multiset where each element has multiplicity one.  For a multiset, $A$, the \textsl{support} of $A$, denoted $S_{A}$, is the set of distinct elements of $A$.  Thus $S_{A} = A \cap [m]$.

A collection of $k$-multisets from $[m]$ is intersecting if every pair of multisets from the collection is intersecting and it is $t$-intersecting if the intersection of every pair of multisets is a multiset with cardinality at least equal to $t$.   
A graph homomorphism was used in~\cite{MR2861399} to prove the next theorem which extends a previous result of Brockman and Kay~\cite{bk}.  This graph homomorphism is used in Section~\ref{int} to prove additional theorems concerning intersecting collections of multisets.

\begin{theorem}[Meagher and Purdy~\cite{MR2861399}]\label{m>k}
Let $k$ and $m$ be positive integers with $m \geq k +1 $.  If $\fim$ is a collection of intersecting $k$-multisets of $[m]$, then \[|\fim| \leq \binom{m + k - 2}{ k-1}.\]  Moreover, if $m >k + 1$, equality holds if and only if $\fim$ is a collection of all the $k$-multisets of $[m]$ that contain a fixed element from $[m]$.
\end{theorem}

A result concerning the size of the largest $t$-intersecting families of $k$-multisets from $[m]$ was recently proved by  F{\"u}redi, Gerbner and Vizer~\cite{FGV} using an operation which they call \textsl{down-compression}.   In the following theorem, $AK(m\! +\! k\!-\!1,\,k,\,t)$ represents the size of the largest $t$-intersecting collection of $k$-subsets from $[m\!+\!k\!-\!1]$ given by Theorem~\ref{AK}.

\begin{theorem}[F{\"u}redi, Gerbner and Vizer~\cite{FGV}]\label{furedi}
Let $1 \leq t \leq k$ and let $2k-t \leq m$.  If $\fim$ is a $t$-intersecting family of $k$-multisets containing elements from $[m]$, then $$\left| \fim \right| \leq AK(m\!+\!k\!-\!1,\,k,\,t)\:.$$
\end{theorem}

Since $m \geq 2k-t$ ensures that $m \geq t+2r$ where $r$ is given by Theorem~\ref{AK} for $n=m+k-1$, this size can be attained by a family of multisets of the form $$\sfm{r}=\left\{G \in \multichoose{[m]}{k}: \,\left| G \cap [t+2r] \right| \geq t+r\right\}.$$  In Section~\ref{t-int}, we use the down-compression operation from~\cite{FGV} and a graph homomorphism to add the structure of the families that attain this size.  

\section{Theorems for set systems}\label{int sets}


In this section, we state some additional known results inspired by the \erd--Ko--Rado theorem.  These theorems are used in Section~\ref{int} to prove analagous results for multisets.

Define an intersecting family of $k$-sets as follows: $$\widetilde{\fis} = \left\{F \in \binom{[n]}{k}: 1 \in F \mbox { and } F \cap [2,k\!+\!1]\neq \emptyset\right\} \cup [2,k\!+\!1].$$ Then  $$| \widetilde{\fis} | = \binom{n-1}{k-1} - \binom{n-k-1}{k-1} +1,$$ where $\binom{n-1}{k-1}$ counts the $k$-subsets of $[n]$ that contain $1$ and $\binom{n-k-1}{k-1}$ counts the $k$-subsets containing $1$ that do not intersect with $[2,k\!+\!1]$.  

In~\cite{MR0219428}, Hilton and Milner considered the size of the largest intersecting family of $k$-subsets of $[n]$ where there is no element common to all subsets.  They proved the following theorem which disproves a conjecture in~\cite{MR0140419}.

\begin{theorem}[Hilton and Milner~\cite{MR0219428}]\label{HM}
Let $k$ and $n$ be positive integers with $2\leq k \leq n/2$.  Let $\fis$ be an  intersecting family of $k$-subsets of $[n]$.  If $$\bigcap_{F \in \fis} F = \emptyset,$$ then $$\left| \fis \right| \leq \binom{n-1}{k-1} - \binom{n-k-1}{k-1}+1.$$
For $3< k < n/2$, only families isomorphic to $\widetilde{\fis}$ will attain this size.
\end{theorem}
Hilton and Milner observed that  $| \widetilde{\fis} | = | \sfs{1} |$ when $k=2$, $k=3$ or $n=2k$.  If $k=2$,  then $\widetilde{\fis} = \sfs{1}$.  However, $\widetilde{\fis}$ and $ \sfs{1}$ are not isomorphic if $k=3$ or $n=2k$.  

In~\cite{HR}, Hajnal and Rothschild generalize the \erd--Ko--Rado theorem to families of $k$-subsets of $[n]$ that are not completely pairwise $t$-intersecting.  These families are defined as having the property that, for $s,t \geq 1$, no more than $s$ of the $k$-subsets have pairwise fewer than $t$ elements in common.  We denote this property by $P(s,t)$.   A family with property $P(s,t)$ can be constructed by choosing $s$ disjoint $t$-subsets of $[n]$ and taking all of the $k$-subsets that contain at least one of these $t$-subsets.  Such a family is said to be \textsl{fixed} by the $t$-subsets. 

\begin{theorem}[Hajnal and Rothschild~\cite{HR}]\label{thmHJ}

Let $\fis$ be a family of $k$-subsets of $[n]$.  Suppose $\fis$ satisfies $P(s,t)$.  Then there is a function $n(k,s,t)$ such that if $n > n(k,s,t)$, $$\left| \fis \right| \leq \sum_{j=1}^{s}(-1)^{j+1}\binom{s}{j}\binom{n-jt}{k-jt}.$$  Equality holds if only if $\fis$ is the family of all $k$-subsets of $[n]$ fixed by some $s$ disjoint $t$-subsets of $[n]$.

\end{theorem}

When $t=1$, the number of $k$-sets of $[n]$ fixed by $s$ disjoint $1$-subsets of $[n]$ is equal to the number of $k$-sets that intersect a given $s$-set.  Thus $$\sum_{j=1}^{s}(-1)^{j+1}\binom{s}{j}\binom{n-j}{k-j}=\binom{n}{k}-\binom{n-s}{k}.$$  If $t=s=1$, a family satisfying $P(s,t)$ is an intersecting family of $k$-subsets and the theorem is equivalent to the \erd--Ko--Rado theorem.   For $t=1$ and $s \geq 1$, a family with property $P(s,t)$ corresponds to the vertex set of an induced subgraph of the Kneser graph that does not contain a clique of size $s+1$.  In particular, a family with property $P(2,1)$ will correspond to a triangle-free subgraph of $K(n,k)$.  

A number of upper bounds for $n(k,s,1)$ when Theorem~\ref{thmHJ} is restricted to $t=1$ have appeared in the literature.  These include $2k^{3}s$ in~\cite{BDE}, $3k^{2}s$ in~\cite{HLS} and $\frac{2k^{2}s}{\log k}$ in~\cite{FLM}.  Most recently, Frankl proved the following theorem.          

\begin{theorem}[Frankl~\cite{MR3033661}] \label{thm:FLM}
Let $\fis$ be a family of $k$-subsets of $[n]$ with property $P(s,1)$.  If $ n \geq (2s\!+\!1)k-s$, then $$ \left | \fis \right| \leq \binom{n}{k} - \binom{n-s}{k}$$ with equality if and only if $\fis$ consists of all $k$-subsets that intersect a given $s$-subset of $[n]$.
\end{theorem}

A slightly different problem is addressed by Frankl and F\"{u}redi in~\cite{FF}.  They prove the following theorem giving the maximum possible size of the union of two intersecting families of $k$-sets of $[n]$ where the two families are not required to be disjoint.  This is equivalent to the number of vertices in the largest induced bipartite subgraph of the Kneser graph $K(n,k)$.  Any set that is in both of the intersecting families can be arbitrarily assigned to one partition or the other and will be an isolated vertex in the induced subgraph.  A bipartite subgraph is triangle-free, that is, the vertices in the subgraph have property $P(2,1)$.  However, property $P(2,1)$ does not imply that the induced subgraph in $K(n,k)$ is bipartite.  

\begin{theorem}[Frankl and F\"{u}redi~\cite{FF}]\label{bipartite}
Let $\fis_{1}$ and $\fis_{2}$ be intersecting families of $k$-subsets of $[n]$.  If $n > \frac{1}{2}(3+\sqrt{5})k$, then $$\left|\fis_{1} \cup \fis_{2}\right| \leq \binom{n-1}{k-1} + \binom{n-2}{k-1}.$$ Equality holds if and only if there exists a $2$-set, $\{x,y\} \subset [n]$, such that $\fis_{1} \cup \fis_{2}$ consists of all the $k$-sets of $[n]$ that intersect $\{x,y\}$.
\end{theorem}

\section{Intersecting multisets}\label{int}


For positive integers $k$ and $m$, let $M(m,k)$ be the graph whose vertices are the $k$-multisets  from $[m]$ and where two vertices are adjacent if and only if the corresponding $k$-multisets are disjoint (i.e.\ the intersection of the two multisets is the empty set).  The graph $M(m,k)$ has $\binom{m+k-1}{k}$ vertices and an independent set of vertices is an intersecting collection of $k$-multisets.   

The following proposition is proved in~\cite{MR2861399}.

\begin{prop}\label{homomorphism}
Let $m$ and $k$ be positive integers and set $n = m+k-1$.  Then there exists a function $f: {[n] \choose k} \rightarrow \multichoose {[m]}{k}$ with the following properties:
\begin{enumerate}
\item $f$ is a bijection.
\item For any $B \in \binom{[n]}{k}$, the support of $f(B)$ is equal to $B \cap [m]$.
\item $f$ is a graph homomorphism from $K(n,k)$ to $M(m,k)$.
\end{enumerate}
\end{prop} 

We will use this graph homomorphism to prove three new results for multisets based on the three theorems for set systems stated in Section~\ref{int sets}.

The next proposition is used in the proof of Theorem~\ref{Hilton Milner}.  The term \textsl{maximal} refers to a collection of multisets, $\fim \subset \multichoose{[m]}{k}$, where $\fim$ is intersecting but $\fim \cup A$ is not  intersecting for any multiset $A \in \binom{[m]}{k}\backslash\fim$.

\begin{prop}\label{support}
Let $\fim$ be a maximal intersecting $k$-multiset system from $[m]$ with $m \geq k+1$.  Then $\fim$ will contain a $k$-multiset with support of cardinality $k$.
\end{prop}
\begin{proof}
For a given $A \in\fim$, let $B$ be any set of $(k - \left| S_{A} \right|)$ elements from $[m]\backslash S_{A}$ and let $C = B \cup S_{A}$.  Then $\left| S_{C}\right| = k$ and $C$ will intersect with all multisets in $\fim$.  Since $\fim$ is maximal, it follows that $C \in \fim$. \qed
\end{proof}

Define an intersecting $k$-multiset system, $\widetilde{\fim}$, as follows:
$$\widetilde{\fim} =\big \{A \in \multichoose{[m]}{k}: 1 \in A \mbox{ and } A \cap [2,k\!+\!1]\neq \emptyset\big\}\cup [2,k\!+\!1].$$  Then $$| \widetilde{\fim} | = \multichoose{m}{k-1} - \multichoose{m-k}{k-1} + 1 = \binom{m+k-2}{k-1} - \binom{m-2}{k-1} + 1.$$  

We now show that the size of the largest intersecting collection of $k$-multisets with no element common to all of the multisets is equal to the size of $\widetilde{\fim}$ when $m \geq k+1$.  (For $2 < m < k+1$, the largest intersecting families with no common element are the largest intersecting families as given by Theorem 1.3 in~\cite{MR2861399}.)

\begin{theorem}\label{Hilton Milner}
Let $k$ and $m$ be positive integers with $1 < k \leq m-1$.  Let $\fim$ be a collection of intersecting $k$-multisets of $[m]$ such that $$\bigcap_{A \in\fim}A = \emptyset.$$  Then \[|\fim| \leq \binom{m + k - 2}{k-1}-\binom{m-2}{k-1}+1 .\]  Moreover, if $3 < k < m-1$, equality holds if and only if $\fim$ is a collection of $k$-multisets from $[m]$ that  is isomorphic to $\widetilde{\fim}$.
\end{theorem}

\begin{proof}
Let $\fim$ be an intersecting family of $k$-multisets from $[m]$ of maximum possible size such that $$\bigcap_{A \in \fim} A= \emptyset.$$  This implies that $$\bigcap_{A \in \fim} S_{A} = \emptyset.$$  Set $n=m+k-1$ and let $f: {[n] \choose k} \rightarrow \multichoose {[m]}{k}$ be a function with the properties given in Proposition~\ref{homomorphism}. Then $\mathcal B= f^{-1}(\fim)$ is an intersecting family of $k$-sets from $[n]$ with $B \cap [m] = S_{f(B)}$ for all $B \in \mathcal B$ and so
\begin{equation}\label{1}
\bigcap_{B \in \mathcal B} (B\; \cap \; [m])  =  \bigcap_{A \in \fim} S_{A}= \emptyset.
\end{equation}
From Proposition~\ref{support}, we know that there exists some $A_{k} \in\fim$ such that $\left| S_{A_{k}}\right| = k$.  Then $f^{-1}(A_{k})= A_{k} \in \mathcal B$.  Since $A_{k} \subset [m]$, it is clear that $A_{k} \cap [m\!+\!1, n] = \emptyset$.  It then follows that $$\bigcap_{B \in \mathcal B}  B \; \cap\;  [m\!+\!1,n] = \emptyset.$$  Combining this with Equation~\ref{1} gives $$\bigcap_{B \in \mathcal B} B   = \emptyset.$$

Since $\mathcal B$ is an intersecting $k$-set system with no element common to all the sets and $n \geq 2k$, the Hilton-Milner Theorem (Theorem~\ref{HM}) gives $$\left| \mathcal B \right| \leq \binom{n-1}{k-1} - \binom{n-k-1}{k-1} + 1 = \binom{m+k-2}{k-1} - \binom{m-2}{k-1} +1 \,.$$  Therefore $$\left| \fim \right| \leq  \binom{m+k-2}{k-1} - \binom{m-2}{k-1} +1 \,.$$  

To prove the uniqueness statement in the theorem, let $m\! >\! k\!+\!1$ and let $\fim$ be an intersecting multiset system of size  $\binom{m+k-2}{k-1} - \binom{m-2}{k-1} +1$.  With the function used above, the preimage of $\fim$ will be an independent set in $K(n,k)$ of size  $\binom{n-1}{k-1} - \binom{n-k-1}{k-1} + 1$.  By Theorem~\ref{HM}, if $k > 3$ then $\mathcal B= f^{-1}(\fim)$ is a set system isomorphic to  $$\widetilde{\fis} = \left\{F \in \binom{[n]}{k}: 1 \in F \mbox { and } F \cap [2,k\!+\!1]\neq \emptyset\right\} \cup [2,k\!+\!1].$$  Specifically, there will exist some set $X \subset [n]$ with $\left | X \right | =k$ and some $y \in [n]\backslash X$ such that  $\mathcal B = \{B \in \binom {[n]}{k} : y \in B \mbox{ and } B \cap X \neq \emptyset\} \cup X$.  Note that $X \subset [n]\backslash \{y\}$ and, since $n=m+k-1$, there must be at least one element from $[m]$ in X.  

Suppose that $y \notin [m]$.  Then $X$ must contain at least two elements from $[m]$.  Call these elements $x_{1}$ and $x_{2}$.  Then the sets $B_{1} = \{x_{1}, m\!+\!1,\dots,n\}$ and $B_{2}=\{x_{2}, m\!+\!1,\dots,n\}$ are in $\mathcal B$.  But $f(B_{i}) \cap [m] = x_{i}$, so $f(B_{1}) \cap f(B_{2}) = \emptyset$ which contradicts our assumption that $\fim$ is intersecting.  Therefore, $y$ must be an element of $[m]$.

Now suppose that $X \not \subset [m]$. Then there exists some $x \in X$ such that $x \in [m\!+\!1,n]$ and the set $B = \{y,m\!+\!1,\dots,n\}$ will be in $\mathcal B$ since $y \in B$ and $B \cap X \neq \emptyset$.  But $B \cap [m]=\{y\}$ so $f(B) \cap f(X) = \emptyset$ which again contradicts the assumption that $\fim$ is intersecting. 

Thus $X \subset [m]$ and $y \in [m]$ and it follows from the properties of $f$ that $$\fim = \{A \in \multichoose{[m]}{k}: y \in A \mbox{ and } A \cap X \neq \emptyset\} \cup X$$ (i.e. $\fim$ will be isomorphic to $\widetilde{\fim}$). \qed 
\end{proof}

If $k\!=\!3$ or $m\!=\!k\!+\!1$, then families attaining the maximum size are not limited to those isomorphic to $\widetilde{\fim}$.  For example, $\sfm{1}$ will attain the maximum size even though it is not isomorphic to $\widetilde{\fim}$.  Specifically, if $k=3$, $$| \widetilde{\fim} | = \left| \sfm{1}\right| = 3m-2\,,$$ and if $m=k+1$, $$| \widetilde{\fim}| = \left| \sfm{1}\right| = \binom{2k-1}{k-1}.$$

  \begin{theorem}\label{thmHJmulti}
Let $\fim$ be a family of $k$-multisets of $[m]$ and suppose that $\fim$ satisfies $P(s,1)$, that is, no more than $s$ of the multisets are pairwise disjoint.  If $m > (2k\!-\!1)s$, then  
$$ \left| \fim \right| \leq \multichoose{m}{k}-\multichoose{m-s}{k}=\binom{m\!+\!k\!-\!1}{k} - \binom{m\!-\!s\!+\!k\!-\!1}{k}.$$
Equality holds if and only if $\fim$ is the family of all $k$-multisets that intersect with a given set of size $s$.

\end{theorem}

\begin{proof}   Let $\fim$ be a family of $k$-multisets from $[m]$ of maximum possible size having property $P(s,1)$ and assume that $m > (2k\!-\!1)s$.  Set $n = m+k-1$ and let $f: {[n] \choose k} \rightarrow \multichoose {[m]}{k}$ be a function with the properties given in Proposition~\ref{homomorphism}.  Then $\fis=f^{-1}(\fim)$ is a family of $k$-sets from $[n]$ with property $P(s,1)$ and, since $m>(2k\!-\!1)s$ implies that $n \geq (2s\!+\!1)k-s$, Theorem~\ref{thm:FLM} gives $$\left| \fis \right| \leq \binom{n}{k}-\binom{n-s}{k}.$$  Thus $$ \left| \fim \right| \leq \binom{n}{k}-\binom{n-s}{k}=\multichoose{m}{k}-\multichoose{m-s}{k}.$$  Since $m > (2k\!-\!1)s \geq s$, a family consisting of all $k$-multisets that intersect with $[s]$ will attain this size and will have property $P(s,1)$.  Since $\fim$ was assumed to be as large as possible, it follows that $$ \left| \fim \right| =\multichoose{m}{k}-\multichoose{m-s}{k}.$$

From Theorem~\ref{thm:FLM}, we know that $f^{-1}(\fim)$ is a family consisting of all $k$-subsets of $[n]$ generated by some $s$ disjoint $t$-subsets of  $[n]$.  Since $t=1$, this is equivalent to the family consisting of all $k$-subsets that intersect with a given $s$-set from $[n]$.  If this $s$-set, $S$, is a subset of $[m]$, then it follows from the definition of $f$ that $\fim$ consists of all $k$-multisets that intersect $S$.  If $S$ is not a subset of $[m]$, then there will be some $y \in S$ such that $y\in [m\!+\!1, n]$.  Thus for all $x \in [m]$, the set $\{x, m\!+\!1, \dots, n\}$ will be in $f^{-1}(\fim)$ and the multisets consisting of $k$ copies of $x$ will be in $\fim$.  Since $m \geq s+1$, this contradicts our assumption that $\fim$ has property $P(s,1)$.  Therefore, if $m > (2k\!-\!1)s$, and $\fim$ is a family of $k$-multisets of $[m]$ with  property $P(s,1)$ having the maximum possible size, then $\fim$ consists of all $k$-multisets that intersect with a given set of size $s$. \qed 
\end{proof}

\begin{theorem}\label{bipartitemulti}
Let $\fim_{1}$ and $\fim_{2}$ be intersecting families of $k$-multisets with elements from $[m]$.  If $m > \frac{1}{2}(1+\sqrt{5})k + 1$, then $$\left| \fim_{1} \cup \fim_{2}\right| \leq \multichoose{m}{k-1} + \multichoose{m-1}{k-1} = \binom{m+k-2}{k-1} + \binom{m+k-3}{k-1}.$$  Equality holds if and only if there exists a $2$-set, $\{x,y\} \subset [m]$, such that $\fim_{1} \cup \fim_{2}$ consists of all $k$-multisets that intersect $\{x,y\}$.
\end{theorem} 

\begin{proof}
Let $\fim_{1}$ and $\fim_{2}$ be intersecting families of $k$-multisets from $[m]$.  Set $n=m+k-1$ and let $f: {[n] \choose k} \rightarrow \multichoose {[m]}{k}$ be a function with the properties given in Proposition~\ref{homomorphism}.  Then  $\fis_{1}=f^{-1}(\fim_{1})$ and $\fis_{2}=f^{-1}(\fim_{2})$ are intersecting families of $k$-sets of $[n]$.  Since $f$ is bijective, if there is any multiset $G$ such that $G \in \fim_{1} \cap \fim_{2}$, then $f^{-1}(G) \in \fis_{1} \cap \fis_{2}$.  Thus $\left| \fim_{1} \cup \fim_{2} \right| = \left| \fis_{1} \cup \fis_{2}\right|$.  Since $m > \frac{1}{2}(1+\sqrt{5})k + 1$ implies $n > \frac{1}{2}(3+\sqrt{5})k$, it follows from Theorem~\ref{bipartite} that $$\left| \fim_{1}\cup \fim_{2}\right| \leq \binom{n-1}{k-1} + \binom{n-2}{k-1} = \multichoose{m}{k-1} +\multichoose{m-1}{k-1}.$$
If $\fim_{1}$ consists of all $k$-multisets containing $1$ and $\fim_{2}$ consists of all $k$-multisets containing $2$ but not $1$, then $$\left| \fim_{1}\cup \fim_{2}\right | = \multichoose{m}{k-1} +\multichoose{m-1}{k-1},$$ the upper bound given in the theorem.  Note that this is equal to the size of the family consisting of all $k$-multisets that intersect with $\{x,y\}$ and that $$  \multichoose{m}{k-1} +\multichoose{m-1}{k-1} = \multichoose{m}{k}-\multichoose{m-2}{k}.$$

To prove the uniqueness portion of the theorem, assume that $\fim_{1} \cup \fim_{2}$ is as large as possible.  Then $$\left|\fis_{1} \cup \fis_{2} \right| =\binom{n-1}{k-1}+\binom{n-2}{k-1},$$ and by Theorem~\ref{bipartite} there exists a set $\{x,y\} \subset [n]$ such that $F\cap \{x,y\} \neq \emptyset$ for all $F \in \fis_{1} \cup \fis_{2}$.  Since the case of $k=1$ is trivial, we will assume that $k \geq 2$.  This implies that $m \geq 3$.  If $\{x,y\} \subset [m]$, then it follows from the definition of $f$ that $\fim_{1} \cup \fim_{2}$ consists of all $k$-multisets that intersect $\{x,y\}$ and we are done, so assume without loss of generality that $y \notin [m]$.  Then $$\left\{\{1,m\!+\!1,\dots,n\},\{2,m\!+\!1,\dots,n\},\{3,m\!+\!1,\dots,n\}\right\} \subset \fis_{1} \cup \fis_{2}$$ and therefore the multisets consisting of $k$ copies of $1$, $2$ and $3$ will be in $\fim_{1} \cup \fim_{2}$.  This contradicts our assumption that $\fim_{1}$ and $\fim_{2}$ are intersecting families.  Thus if $\fim_{1}$ and $\fim_{2}$ are intersecting families of $k$-multisets from $[m]$ such that $\fim_{1} \cup \fim_{2}$ is as large as possible, then $\{x,y\} \subset [m]$ and $\fim_{1} \cup \fim_{2}$ consists of all $k$-multisets that intersect $\{x,y\}$. \qed
\end{proof}

\section{$t$-intersecting multisets}\label{t-int}


The main result in this section is the following theorem which is a restatement of Theorem~\ref{furedi} with the addition of the structure of the $t$-intersecting multiset families that attain the maximum size.  Recall that $\sfm{r}$ is defined as $$\sfm{r}=\left\{G \in \multichoose{[m]}{k}: \,\left| G \cap [t+2r] \right| \geq t+r\right\}.$$

\begin{theorem}\label{our t result}
Let $1 \leq t \leq k$ and let $r$ be a non-negative integer such that $r \leq k-t$. Let $\fim$ be a $t$-intersecting collection of $k$-multisets of $[m]$ with $m \geq 2k-t$.
\begin{enumerate}
\item If \[(k-t+1)\left(2 + \frac{t-1}{r+1}\right) <m\!+\!k\!-\!1 < (k-t+1)\left(2 + \frac{t-1}{r}\right),\] then $\left| \fim \right| \leq \left | \sfm{r}\right|$.   (By convention, $\frac{t-1}{r} = \infty$ for $r=0$.)

 If $r > 0$, equality holds if and only if $\fim$ is isomorphic to $\sfm{r}$.  If $r=0$, equality holds if and only if $\fim$ consists of all $k$-multisets containing a fixed $t$-multiset.
 
\item If $$n=(k-t+1)\left(2 + \frac{t-1}{r+1}\right),$$ then $\left| \fim \right| \leq \left | \sfm{r}\right| = \left| \sfm{r+1} \right|$. 

If $ r>0$, equality holds if and only if $\fim$ is isomorphic to either $\sfm{r}$ or $\sfm{r+1}$.  If $r=0$, equality holds if and only if $\fim$ is isomorphic to $\sfm{1}$ or it consists of all $k$-multisets containing a fixed $t$-multiset.
\end{enumerate}  
\end{theorem}

Our proof uses a graph homomorphism between two graphs having the same vertex sets as $K(n,k)$ and $M(m,k)$  but where independent sets of vertices are $t$-intersecting families of $k$-sets and $k$-multisets respectively.  We begin by defining these graphs and showing that a function with the properties given in Proposition~\ref{homomorphism} exists.

 Let $K(n,k,t)$ be the graph whose vertices are the $k$-subsets of $[n]$ and where two vertices, $A$ and $B$, are adjacent if and only if $\left| A\cap B\right| <t$.  An independent set of vertices in $K(n,k,t)$ will be a $t$-intersecting family of $k$-subsets of $[n]$. 

Let $M\pr(m,k,t)$ be the graph whose vertices are the $k$-multisets of $[m]$ and where two vertices, $A$ and $B$, are adjacent if and only if $\left| A \cap B \cap [m] \right| < t$.  In other words, two vertices are adjacent if and only if their supports have fewer than $t$ elements in common.  An independent set will be a $t$-intersecting family of $k$-multisets of $[m]$, although not all $t$-intersecting families will be independent sets in $M\pr(m,k,t)$.

\begin{prop}\label{t-int hom}
Let $m$ and $k$ be positive integers and set $n=m+k-1$.  Then there exisits a function $f:{[n] \choose k} \rightarrow \multichoose {[m]}{k}$ with the following properties: 
\begin{enumerate}
\item $f$ is a bijection.
\item For any $A \in \binom{[n]}{k}$, the support of $f(A)$ is equal to $A \cap [m]$.
\item $f$ is a graph homomorphism from $K(n,k,t)$ to $M\pr(m,k,t)$.
\end{enumerate}
\end{prop} 

\begin{proof}
The first two properties follow from Proposition~\ref{homomorphism}. Let $A_{1},A_{2}$ be two adjacent vertices in $K(n,k,t)$.  Then $$\left| A_{1} \cap A_{2} \right| < t$$ and it follows from Property 2 that $$\left| S_{f(A_{1})} \cap S_{f(A_{2})} \right| < t.$$  Thus $$\left| f(A_{1}) \cap f(A_{2}) \cap [m] \right| < t$$ and $f(A_{1})$ is adjacent to $f(A_{2})$ in $M\pr(m,k,t)$.  Hence the bijection $ f:  {[n] \choose k} \rightarrow \multichoose {[m]}{k}$ is a graph homomorphism from $K(n,k,t))$ to $M\pr(m,k,t)$. \qed
\end{proof}

If $f$ is a function with the properties given in Propostion~\ref{t-int hom} and $\fim$ is an independent set in $M\pr(m,k,t)$, then $f^{-1}(\fim)$ will be an independent set in $K(n,k,t)$.  Thus  Theorem~\ref{AK} gives an upper bound on the size of the largest independent set in $M\pr(m,k,t)$.  

We also make use of the down-compression operation defined in~\cite{FGV}.  This operation maintains or increases the size of the supports of the multisets in a $t$-intersecting family of multisets while maintaining the size and $t$-intersection of the family.  Repeated applications of this operation will transform any $t$-intersecting family of multisets into a $t$-intersecting family of the same size but with the additional property that $\left| F_{1} \cap F_{2} \cap [m] \right| \geq t$ for all $F_{1},F_{2}$ in the family.  We begin with some notation and definitions of concepts from~\cite{FGV}.

Let $\fim \subseteq \multichoose {[m]}{k}$ be $t$-intersecting.  A multiset $T$ is called a \textsl{$t$-kernel} for $\fim$ if $\left | F_{1} \cap F_{2} \cap T \right| \geq t$ for all $F_{1},F_{2} \in \fim$.  Note that the multiset containing $t$ copies of each of the integers in $[m]$ will be a $t$-kernel for any $t$-intersecting family.  Let $\mathcal K(\fim)$ denote the set of all $t$-kernels for $\fim$ that contain $[m]$.  For $T \in \mathcal K(\fim)$, let $T_{>1}=T\backslash[m]$.   Thus $T_{>1}$ contains only elements that occur more than once in $T$.

A multiset can be represented by a set of ordered pairs where the first integer is an element of $[m]$ and the second integer serves to differentiate among the copies of that element in the multiset.  For example, the multiset $\{1,1,1,4\}$ can be written as $\{(1,1),(1,2),(1,3),(4,1)\}$.  This ordered pair notation is used in the definition of the following shifting operation.

For $s \leq \mathrm{m}(i,F)$ and $j \in [m]$ with $j \notin F$, let $F_{(i,s)(j)}$ denote the multiset formed from $F$ by replacing all but $s-1$ of the copies of $i$ with $j$.  That is, 
\begin{multline*} F_{(i,s)(j)} = (F\backslash\{(i,x): x \in [s, \mathrm{m}(i,F)]\})\\ \cup 
\{(j,y):  y \in [1,(\mathrm{m}(i,F)-s+1)]\}.
\end{multline*}
 If $s > \mathrm{m}(i,F)$, then $F_{(i,s)(j)}=F$.   F\"{u}redi et al.~\cite{FGV} define a shifting operation on $\fim$ as follows: 

\begin{align*}
\mathcal S_{(i,s)(j)}(\fim)&=\{\mathcal S_{(i,s)(j)}(F) : \: F \in \fim\}, \\
\intertext{where}
\mathcal S_{(i,s)(j)}(F)&=\begin{cases} F_{(i,s)(j)} \: &\mbox { if } j \notin F \mbox { and } F_{(i,s)(j)} \notin \fim ,\\
F \: &\mbox { otherwise}.
\end{cases}
\end{align*}

Provided that $s \geq 2$, this operation increases the size of the support of any multiset $F$ such that $\mathcal S_{(i,s)(j)}(F) \neq F$.  For example, if $F=\{1,1,1,3,4\}$ then $S_{(1,2)(2)}(F) = \{1,2,2,3,4\}$ and the size of the support increases by one.

The down-compression operation consists of combining this shifting operation with the concept of $t$-kernels.  Given a $t$-intersecting family of $k$-multisets from $[m]$ and a $t$-kernel, $T$, such that $[m] \subseteq T$, the down-compression operation consists of sequentially applying the shifting operation for some $i \in [m]$ such that $i \in T_{>1}$ and for all values of $j$ from $1$ to $m$ with $s=\mathrm{m}(i,T)$.   The next lemma shows that applying the down-compression operation gives a $t$-intersecting family of $k$-multisets of the same size as the original family and proves that the new family will have $T\backslash\{i,\mathrm{m}(i,T)\}$ as a $t$-kernel.

\begin{theorem}[F{\"u}redi, Gerbner and Vizer~\cite{FGV} ]\label{down comp}
Let $t$, $k$ and $m$ be positive integers such that $k \geq t$ and $m \geq 2k-t$.  Let $\fim$ be a $t$-intersecting family of $k$-multisets of $[m]$ and let $T \in \mathcal K(\fim)$ with $T_{>1} \neq \emptyset$.  For $i \in T_{>1}$, set $s=\mathrm{m}(i,T)$ and define $$\widehat{\fim} = \mathcal S_{(i,s)(m)}\bigg[\mathcal S_{(i,s)(m-1)}\Big[...\big[\mathcal S_{(i,s)(1)}(\fim)\big]\Big]\bigg].$$  Then:
\begin{enumerate}
\item
 $\widehat{\fim}$ is a $t$-intersecting family of $k$-multisets of $[m]$ and $| \widehat{\fim} | = \left|\fim\right|$, 
\item
 $T \backslash \{i,\mathrm{m}(i,T)\} \in \mathcal K(\widehat{\fim})$.
\end{enumerate}
\end{theorem}

The second statement means that the down-compression operation can be performed on $\widehat{\fim}$ using $T \backslash \{i,\mathrm{m}(i,T)\}$ as the $t$-kernel.  Repeated applications of the operation results in a family that has $[m]$ as a $t$-kernel.  Thus any $t$-intersecting collection of $k$-multisets of $[m]$ can be transformed into a collection of the same size in which $\left| F_{1} \cap F_{2} \cap [m] \right| \geq t$ for all $F_{1},F_{2}$ in the family.  This combined with the next theorem is sufficient to prove Theorem~\ref{furedi}.  (Recall that $AK(m+k-1,k,t)$ denotes the size of the largest $t$-intersecting $k$-subset system from $[m+k-1]$ given by Theorem~\ref{AK}.)

\begin{theorem}[F{\"u}redi, Gerbner and Vizer~\cite{FGV} ]\label{furedi 2}
Let $1 \leq t \leq k$ and let $2k-t \leq m$.  If $\fim$ is a $t$-intersecting family of $k$-multisets of $[m]$ such that $\left| F_{1} \cap F_{2} \cap [m]\right|\!~\geq\!~t$ for all $F_{1},F_{2} \in \fim$, then $$\left| \fim \right| \leq AK(m\!+\!k\!-\!1,k,t).$$
\end{theorem}

We use a similar approach to prove Theorem~\ref{our t result}.  We first prove the result for families where  $\left| F_{1} \cap F_{2} \cap [m] \right| \geq t$ for all $F_{1},F_{2}$ in the family and then show that this can be extended to all $t$-intersecting families using the down-compression operation.

\begin{lemma}\label{structure}
Let $1 \leq t \leq k$ and let $r$ be a non-negative integer such that $r \leq k-t$. For $m\geq 2k-t$, let $\fim$ be a collection of $k$-multisets of $[m]$ such that $\left| F_{1} \cap F_{2} \cap [m]\, \right| \geq t$ for all $F_{1},F_{2} \in \fim$.

\begin{enumerate}
\item If \[(k-t+1)\left(2 + \frac{t-1}{r+1}\right) <m\!+\!k\!-\!1 < (k-t+1)\left(2 + \frac{t-1}{r}\right),\] then $\left| \fim \right| \leq \left | \sfm{r}\right|$.   (By convention, $\frac{t-1}{r} = \infty$ for $r=0$.)
Equality holds if and only if $\fim$ is isomorphic to $\sfm{r}$.
 
\item If $$m+k-1=(k-t+1)\left(2 + \frac{t-1}{r+1}\right),$$ then $\left| \fim \right| \leq \left | \sfm{r}\right| = \left| \sfm{r+1} \right|$. 
 Equality holds if and only if $\fim$ is isomorphic to either $\sfm{r}$ or $\sfm{r+1}$.
\end{enumerate}

 \end{lemma}

\begin{proof}
If $\fim$ satisfies the conditions in the statement of the lemma, then $\fim$ is an independent set in the graph $M\pr(m,k,t)$.  Set $n=m+k-1$ and let  $f: {[n] \choose k} \rightarrow \multichoose {[m]}{k}$ be a function with the properties given in Proposition~\ref{t-int hom}. Then $\fis=f^{-1}(\fim)$ is an independent set in $K(n,k,t)$ and so applying Theorem~\ref{AK} gives $\left| \fis\right| \leq \left| \sfs{r} \right|$ for some $r \in [0,(k-t)]$.  Thus $\left| \fim \right| \leq \left| \sfs{r} \right|$ and since $\left| \sfs{i} \right| = \left| \sfm{i} \right| $ whenever both are defined, it follows that $\left| \fim \right| \leq \left| \sfm{r} \right|$. 

We now consider the structure of the families attaining the maximum possible size in the two statements in the lemma.
\begin{enumerate}
\item Let $(k-t+1)\left(2 + \frac{t-1}{r+1}\right) <m\!+\!k\!-\!1 < (k-t+1)\left(2 + \frac{t-1}{r}\right)$.\\
By Theorem~\ref{AK}, if $\left| \fis \right| = \left| \sfs{r}\right|$, then $\fis$ is isomorphic to $\sfs{r}$.  That is, $\fis$ will consist of all $k$-subsets of $[n]$ containing at least $t+r$ elements of $X$, where $X$ is a subset of $[n]$ of size $t+2r$.  If $X \subseteq [m]$, it follows from the definition of the function $f$ that every multiset in $\fim$ will contain $t+r$ elements of $X$.  Thus $\fim$ is isomorphic to $\fim_{(r)}$.  If $X \nsubseteq [m]$, then there is some $x \in X$ such that $x \notin [m]$.  Since $f^{-1}(\fim)$ will consist of all $k$-subsets of $[n]$ containing at least $t+r$ elements of $X$, there will be a pair of sets $A_{1},A_{2} \in f^{-1}(\fim)$ with $x \in A_{1} \cap A_{2}$ such that $\left| A_{1} \cap X \right| = \left | A_{2} \cap X \right| = t+r$ and $\left| A_{1} \cap A_{2} \right| =t$.  Then $\left| A_{1} \cap A_{2} \cap [m]\,\right| <t$ and so $\left| f(A_{1}) \cap f(A_{2}) \cap [m] \,\right| < t$.   Thus $f(A_{1})$ is adjacent to $f(A_{2})$ in $M\pr(m,k,t)$ which contradicts our assumption that $\fim$ is an independent set in $M\pr(m,k,t)$.  Therefore, any family satisfying the conditions in the lemma is isomorphic to $\fim_{(r)}$.

\item Let $m+k-1=(k-t+1)\left(2 + \frac{t-1}{r+1}\right)$.\\
Again using Theorem~\ref{AK} we have that $ \fis$ attains the maximum possible size if and only if $\fis$ is isomorphic to either $\sfs{r}$ or $\sfs{r+1}$. The argument used in the first case can be used to show that $\fim$ must be isomorphic to either $\sfm{r}$ or $\sfm{r+1}$.  \qed
\end{enumerate}
\end{proof}

Lemma~\ref{structure} along with Theorem~\ref{down comp} are sufficient to prove the size of the largest $t$-intersecting families as given in Theorem~\ref{our t result}.  The next two lemmas are needed to extend the structure portion of Lemma~\ref{structure} to all $t$-intersecting multiset families.  The first deals with the case when down-compression of the family of maximum size results in a family isomorphic to $\sfm{r}$ for some $r>0$.  The second deals with the case when down-compression of the family results in a family isomorphic to $\sfm{0}$.  The need to treat $r=0$ separately arises from the fact that a family consisting of all $k$-multisets containing a fixed $t$-multiset will have the same cardinality regardless of the multiplicity of each of the elements in the $t$-multiset.  Thus there are families of multisets that have the same size as $\sfm{0}$ but that are not isomorphic to $\sfm{0}$.  For example, if $m=5$, $k=4$ and $t=2$, the family consisting of all $4$-multisets containing $\{1,1\}$ is not isomorphic to $\sfm{0}$ but has the same size as $\sfm{0}$, the family consisting of all $4$-multisets containing $\{1,2\}$.  This does not occur when $r>0$.  If $r>0$, a family consisting of all $k$-multisets containing at least $t+r$ elements from a $(t\!+\!2r)$-multiset in which the multiplicity of some element is greater than one contains fewer $k$-multisets than $\sfm{r}$.  For instance, if $m=5$, $k=4$ and $t=2$, then $$\sfm{1}=\left\{A \in \multichoose{[5]}{4}: \,\left| A \cap [4] \right| \geq 3\right\}$$ and so $$\left|\sfm{1} \right| = \binom{4}{4} + \binom{4}{3}\multichoose{4}{1}=17.$$ 
Now consider the family of multisets defined as follows.  $$\fim =\left\{A \in \multichoose{[5]}{4}: \, \left| A \cap \{1,1,2,3\}\right| \geq 3\right\}.$$  Then $$\left| \fim \right| = \binom{4}{4} + 3\multichoose{4}{1} = 13$$  since there are only three distinct ways to select three elements from $\{1,1,2,3\}$.  

Recall that $\widehat{\fim}$ is the family of $k$-multisets formed by applying the down-compression operation to $\fim$, a $t$-intersecting family of $k$-multisets from $[m]$.  It is defined as  
 $$\widehat{\fim} = \mathcal S_{(i,s)(m)}\bigg[\mathcal S_{(i,s)(m-1)}\Big[...\big[\mathcal S_{(i,s)(1)}(\fim)\big]\Big]\bigg]$$ where $i \in T_{>1}$ for some $t$-kernel $T \in \mathcal K(\fim)$ and $s=\mathrm{m}(i,T)$.

\begin{lemma}\label{downcomp maintains structure}
Let $1 \leq t \leq k$ and $m \geq 2k-t$.  Let $\fim \subseteq \multichoose{[m]}{k}$ be a $t$-intersecting family of multisets.  For $T \in \mathcal K(\fim)$ and $i \in [m]$ such that $i \in T_{>1}$, set $s=m(i,T)$.  If $\widehat{\fim}$ is isomorphic to $\fim_{(r)}$ for some $r>0$, then so is $\fim$.

\end{lemma}

\begin{proof}
It is sufficient to show that if $\mathcal S_{(i,s)(j)}(\fim)$ is isomorphic to $\fim_{(r)}$ for some $r>0$, then so is $\fim$.  Let $\mathcal S_{(i,s)(j)}(\fim)=\fim\pr$.  We assume without loss of generality that $\fim\pr = \fim_{(r)}$.  Thus $$\fim\pr = \left\{ B \in \multichoose{[m]}{k}: \left| B \cap [t\!+\!2r]\,\right|\geq t+r\right\}.$$ 

If $B \in \fim\pr$ and $j \notin B$ or $\mathrm{m}(i,B) \neq s-1$, then it follows from the definition of the down-compression operation that $B \in \fim$. Therefore, in the remainder of the proof we assume that $j \in B$ and $\mathrm{m}(i,B) = s-1$ for some $B \in \fim\pr$.  With these assumptions, either $B \in \fim$ or there is some other multiset $A \in \fim$ such that $S_{(i,s)(j)}(A)= B$.  Then $$A = (B\backslash\{(j,1),\dots,(j,\mathrm{m}(j,B))\}) \cup \{(i,s),\dots,(i,s\!+\!\mathrm{m}(j,B)\!-\!1)\}.$$  Note that $j \notin A$ and that $\mathrm{m}(i,A) \geq s$. We consider two cases and show that in both $B \in \fim$.  

\begin{enumerate}
\item Case 1:  $j \notin [t+2r]$.

Then $\left| A \cap [t\!+\!2r]\, \right| = \left| B \cap [t\!+\!2r]\,\right| \geq t+r$ and so $A \in \fim\pr$.  But $j \notin A$ and $\mathrm{m}(i,A) \geq s$, so $A \in \fim\pr$ only if $B \in \fim$.  Thus $B \in \fim\pr$ implies that $B \in \fim$. 

\item Case 2: $j \in [t+2r]$.

If $\left| B \cap [t\!+\!2r] \, \right| >t+r$, then $\left| A \cap [t\!+\!2r]\, \right| \geq t+r$.  Thus $A \in \fim\pr$ which implies that $B \in \fim$.

Now suppose that $\left| B \cap [t\!+\!2r] \, \right| =t+r$.  We can construct a multiset $B\pr$ such that $\left| B\pr \cap [t\!+\!2r] \,\right| = t+r$ and $\left| B \cap B\pr \cap [t\!+\!2r] \,\right|=t$ with $j \in B \cap B\pr$.  We can further stipulate that $\mathrm{m}(j,B\pr)=k-t-r+1$ and that $i \notin B\pr$.  This last condition is possible even if $i \in [t\!+\!2r]$ since $r>0$ and $i \in B$.  Clearly $B\pr \in \fim\pr$ and since $\mathrm{m}(i,B\pr) \neq s-1$, it must also be in $\fim$.  But $\left| B\pr \cap A \right| = t-1$ and so $A$ cannot be in $\fim$.  Therefore $B \in \fim$.
\end{enumerate}
Thus, if $B$ is an element of $\fim\pr$ it is also in $\fim$.   Since $\left| \fim\pr \right| = \left| \fim \right|$ by Lemma~\ref{down comp}, it follows that $\fim\pr =\fim $. \qed
\end{proof}

\begin{lemma}\label{structure if r=0}
Let $1 \leq t \leq k$ and $m \geq 2k-t$.  Let $\fim \subseteq \multichoose{[m]}{k}$ be a $t$-intersecting family.  For $T \in \mathcal K(\fim)$ and $i \in [m]$ such that $i \in T_{>1}$, set $s=m(i,T)$.  If $\widehat{\fim}$ is a collection of all of the $k$-multisets of $[m]$ containing a fixed $t$-multiset, then $\fim$ is a collection of all of the $k$-multisets of $[m]$ containing a fixed $t$-set.  

\end{lemma}

\begin{proof} It is sufficient to show that $S_{(i,s)(j)}(\fim)$ has a $t$-multiset common to all of its $k$-multisets only if $\fim$ does.  If $k=t$, the result is trivial so we assume that $k >t$.    

Let $S_{(i,s)(j)}(\fim)= \fim\pr$.  Assume that there exists a $t$-multiset, $X \in \multichoose{[m]}{t}$, such that $$\fim\pr = \left\{ B \in \multichoose {[m]}{k} : X \subseteq B\right\}\,.$$  As in the proof of Lemma~\ref{downcomp maintains structure}, we assume that $j \in B$ and $\mathrm{m}(i,B) = s-1$ for some $B \in \fim\pr$.  Then $\mathrm{m}(i,X) \leq s-1$ and either $B \in \fim$ or $$A = (B\backslash\{(j,1),\dots,(j,\mathrm{m}(j,B))\}) \cup \{(i,s),\dots,(i,s\!+\!\mathrm{m}(j,B)\!-\!1)\}\in \fim.$$

\begin{enumerate}
\item Case 1:  $j \notin X$.

Since $\mathrm{m}(i,A) > \mathrm{m}(i,B)$, it follows that $X \subset A$ and so $A \in \fim\pr$ and $A \in \fim$.  This implies that $B \in \fim$.  Therefore $\fim= \fim\pr$.  

\item Case 2:  $j \in X$ and $\mathrm{m}(i,X) \neq s-1$.

Let $B\pr = X \cup \{(j,\mathrm{m}(j,X)+1),\dots,(j,k\!-\!t\!+\!\mathrm{m}(j,X))\}$.  Then $B\pr \in \fim\pr$ and $\mathrm{m}(i,B\pr) \neq s-1$.  Thus $B\pr \in \fim$.  But $j \notin A$ and so $\left| A \cap B\pr \right| <t$.  Since $\fim$ is $t$-intersecting, it follows that $A \notin \fim$ and therefore $B \in \fim$.  Thus $\fim= \fim\pr$. 

\item Case 3:  $j \in X$ and $\mathrm{m}(i,X) = s-1$.

Partition $\fim\pr$ as follows.

Let
\begin{align*} \fim\pr_{1} &=\{B \in \fim\pr: \mathrm{m}(i,B)=s-1\}
\intertext {and let}
 \fim\pr_{0}&= \fim\pr\backslash \fim\pr_{1}.
\end{align*}
Then $\fim\pr_{0} \subset \fim$ since the multiplicity of $i$ will be greater than $s-1$ in all multisets in $\fim\pr_{0}$. 

We now show that either all of the multisets in $\fim\pr_{1}$ are in $\fim$ or all were changed by the down-compression operation.  Define a graph, $G$, with the multisets in $\fim\pr_{1}$ as its vertex set and where two vertices are adjacent if and only if they exactly $t$-intersect.  In other words, two vertices, $B_{1}$ and $B_{2}$ , are adjacent if and only if $B_{1} \cap B_{2} = X$.  Then if $B_{1}$ and $B_{2}$ are adjacent vertices in $G$, either both are in $\fim$ or both were changed by the down-compression operation since if only one changed $\fim$ would not $t$-intersect.  From this it follows that if $G$ is connected, either all multisets in $\fim\pr_{1}$ are in $\fim$ or all were changed by the down-compression operation.

Let $B_{1}$ and $B_{2}$ be any two multisets from $\fim\pr_{1}$.  Since $k > t$ and $m \geq 2k-t$, it follows that $m > k$, and thus there is some $x \in [m]$ such that $x \notin B_{1}$.  Let $C_{1}\!=\!X\, \cup \{(x,1),\dots,(x,k-t)\}$. Then $C_{1}\! \in\! \fim\pr_{1}$ and $B_{1} \cap C_{1} = X$.  Therefore $B_{1}$ is adjacent to $C_{1}$ in $G$.   If $x \notin B_{2}$, then $B_{2}$ is adjacent to $C_{1}$ and thus there is a path between $B_{1}$ and $B_{2}$.  Suppose that $x \in B_{2}$. Then there is some $y \neq x \in [m]$ such that $y \notin B_{2}$.  Let $C_{2} = X \,\cup \,\{(y,1),\dots,(y,k-t)\}$.   Then $B_{2}$ and $C_{1}$ are adjacent to $C_{2}$ and thus there is a path between $B_{1}$ and $B_{2}$.  Since $B_{1}$ and $B_{2}$ were arbitrary vertices of $G$, it follows that $G$ is connected.  Therefore either all of the multisets in $\fim\pr_{1}$ are in $\fim$ and $\fim=\fim\pr$ or all were changed. 

Clearly, if all the multisets in $\fim\pr_{1}$ are in $\fim$ then $\fim=\fim\pr$.  If all of the multisets in $\fim\pr_{1}$ were changed by the shifting operation, we claim that $\fim$ consists of all of the $k$-multisets containing the $t$-multiset $(X\backslash\{(j,1)\}) \cup \{(i, \mathrm{m}(i,X)\!+\!1)\}$.  

Suppose that all of the multisets in $\fim\pr_{1}$ were changed by the shifting operation.  Let $\fim_{1} \subset \fim$ be the collection of multisets such that $S_{(i,s)(j)}(\fim_{1}) = \fim\pr_{1}$.   Choose $B \in \fim\pr_{1}$ such that $j \notin B\backslash X$ and let $A$ be the multiset in $\fim_{1}$ such that $S_{(i,s)(j)}(A) = B$.  Then $j \notin A$ and $\mathrm{m}(i,A) = \mathrm{m}(i,X) + \mathrm{m}(j,X) \geq \mathrm{m}(i,X) + 1$.   Let $$C= X \cup \{(i,s)\} \cup \{(j,\mathrm{m}(j,X)+1),\dots,(j,k-t+\mathrm{m}(j,X)-1)\}.$$ Then $C \in \fim\pr_{0} \subset \fim$ and so $\left| A \cap C \right| \geq t$ since $\fim$ is $t$-intersecting.   But $\left| A \cap C \right| = t - \mathrm{m}(j,X) + 1$, so $\mathrm{m}(j,X)=1$.  It follows that $(X\backslash\{(j,1)\}) \cup \{(i, \mathrm{m}(i,X)\!+\!1)\} \subset A$ for all $A \in \fim_{1}$.   That $(X\backslash\{(j,1)\}) \cup \{(i, \mathrm{m}(i,X)\!+\!1)\} \subseteq A$  for all $A \in \fim\pr_{0} $ follows from the fact that $\mathrm{m}(i,A) > s-1$.
\end{enumerate}
Thus if $S_{(i,s)(j)}(\fim)$ consists of all the $k$-multisets of $[m]$ containing some $t$-multiset $X$, then $\fim$ consists of all the $k$-multisets of $[m]$ containing $X$ or $\fim$ consists of all the $k$-multisets of $[m]$ containing the $t$-multiset $(X\backslash\{(j,1)\}) \cup \{(i, \mathrm{m}(i,X)\!+\!1)\}$. \qed
\end{proof}

The homomorphism of $K(n,k,t)$ to $M\pr(m,k,t)$ can be used in conjunction with the down-compression operation and known results for set systems to prove a $t$-intersecting versions of Theorem~\ref{Hilton Milner}.   Let
 $$\widetilde{\fis}_{t} = \left\{F \in \binom{[n]}{k}: [t] \in F \mbox { and } F \cap [t\!+\!1,k\!+\!1] \neq \emptyset\right\} \cup \bigg\{ \bigcup_{i \in [t]}[k+1]\backslash \{i\}\bigg\}.$$ 
Then $\widetilde{\fis}_{t}$ is a $t$-intersecting family of $k$-subsets of $[n]$ which does not have a $t$-set common to all $k$-subsets.

 We use the following theorem.  

\begin{theorem}[Ahlswede and Khachatrian~\cite{MR1405994}]\label{HM t>1}
Let $ 1 <t < k$ and let $\fis$ be a $t$-intersecting family of $k$-sets from $[n]$ such that $$\left|\bigcap_{F\in\fis} F\, \right| < t.$$ 

\begin{enumerate}
\item If $(t+1)(k-t+1) < n$ and $k \leq 2t+1$, then $\left| \fis \right| \leq \left| \fis_{1}\right|$ and only families isomorphic to $\sfs{1}$ will attain this size.
\item If $(t+1)(k-t+1) < n$ and $k > 2t+1$, then $\left| \fis \right| \leq \max\left\{ \left|\fis_{1}\right|, | \widetilde{\fis_{t}}|\right\}$ and only families isomorphic to $\sfs{1}$ or $\widetilde{\fis_{t}}$ will attain this size.
\end{enumerate}
\end{theorem}

We define a $k$-multiset system, $\widetilde{\fim}_{t}$, as follows:
$$\widetilde{\fim}_{t}= \{ A \in \multichoose{[m]}{k} : [t] \subseteq A, A \cap [t\!+\!1,k\!+\!1] \neq \emptyset\} \cup \{[k\!+\!1]\backslash\{i\}, i=1,\dots,t\}.$$  Then $\widetilde{\fim}_{t}$ is a $t$-intersecting collection of $k$-multisets of $[m]$ which does not have a $t$-multiset common to all $k$-multisets.

\begin{theorem}\label{thm:non-trivial t-int}
For $1<t < k$ and $m \geq 2k -t$, let $\fim$ be a $t$-intersecting family of $k$-multisets from $[m]$ such that $$\left| \bigcap_{A \in\fim}A\; \right| < t.$$

\begin{enumerate}
\item  If $m > t(k-t)+2$ and $k > 2t+1$, then $\left| \fim \right| \leq \left| \sfm{1} \right|$.  Equality holds if and only if $\fim$ is isomorphic to $\sfm{1}$.  

\item  If $m > t(k-t)+ 2$ and $k >2t+1$, then $\left| \fim \right| \leq \max \{ \left| \sfm{1} \right|, | \widetilde{\fim}_{t} |\}$.  Equality holds if and only if $\fim$ is isomorphic to the larger of $\sfm{1}$ and $\widetilde{\fim}_{t}$.
\end{enumerate}

\end{theorem}

When $m \geq 2k-t$, the down-compression operation can be used to transform any $t$-intersecting family of $k$-multisets, $\fim$, into one in which $\left| F_{1} \cap F_{2} \cap [m] \right| \geq t$ for all $F_{1},F_{2} \in \fim$.  If the condition that $\left| F_{1} \cap F_{2} \cap [m] \right| \geq t$ for all $F_{1},F_{2} \in \fim$ is added to Theorem~\ref{thm:non-trivial t-int}, the result can be proved using Theorem~\ref{HM t>1} and the graph homomorphism from $K(n,k,t)$ to $M\pr(m,k,t)$ in a manner similar to the proof of Theorem~\ref{Hilton Milner}.  By Lemma~\ref{downcomp maintains structure}, we know that the down-compression operation will give a family isomorphic to $\sfm{1}$ only if the original family was isomorphic to $\sfm{1}$.  It only remains to be shown that a similar result holds when the down-compressed family is isomorphic to $\widetilde{\fim}_{t}$.  This can be proved using the definition of the down-compression operation and the fact that $\fim$ is $t$-intersecting.  The proof is similar to that of Lemma~\ref{downcomp maintains structure} but requires the consideration of more cases.

\section{Open questions}


The graph homomorphism method used in this paper provides a convenient way to prove results for intersecting families of multisets based on known results for sets, but there are questions that cannot be approached in this manner.  For $t>1$, the graph homomorphism requires that the supports of the multisets $t$-intersect.  Althought the down-compression operation can be used to transform any $t$-intersecting family into one that satisfies this condition, it requires $m\geq 2k-t$.  Also, the graph homomorphism cannot be used to give a multiset version of Theorem~\ref{thmHJ} for $t>1$ since the down-compression operation does not preserve the $P(s,t)$ property. 

When $n \leq 2k-t$, the collection of all $k$-subsets of $[n]$ is $t$-intersecting.  No similar result holds for collections of $k$-multisets of $[m]$ for any $m >1$.  If $m=2$, it is easy to see that any $t$-intersecting family will have a $t$-multiset common to all multisets in the family.  Thus the question of the size and structure of the largest $t$-intersecting family of $k$-multisets of $[m]$ with $t >1$ is open for $2 < m < 2k-t$.  We conjecture that, for all values of $m$, $k$ and $t$, the largest family is the collection of all $k$-multisets containing at least $t+r$ elements from a $(t\!+\!2r)$-multiset where $0 \leq r \leq k-t$.  We further conjecture that the cardinality of the support of the $(t\!+\!2r)$-multiset is equal to the lesser of $m$ and $t+2r$ and that the multiplicities of the elements in the $(t\!+\!2r)$-multiset will vary by at most one unless the multiplicity of all of the elements is greater than or equal to $r$.

If we consider only those $t$-intersecting families of multisets where the supports are $t$-intersecting, the graph homomorphism method gives an upper bound on the size of the largest family for all $m > k-t+1$.  However, this bound may not be attainable.  For example, consider the case when $k=4$ and $t=2$.  Let $\fim$ be a $t$-intersecting family of multisets such that $\left| F_{1} \cap F_{2} \cap [m] \right| \geq t$ for all $F_{1},F_{2} \in \fim$.  If $m=5$, the graph homomorphism method gives $| \sfs{1} |$ as an upper bound on the size of $\fim$ and this size is attained by $\sfm{1}$.  If $m=4$, then $| \fim | \leq | \sfs{2} |$ but a family isomorphic to $\sfm{2}$ cannot be constructed since  $m < t+2r$.

Another obvious question arises from Theorem~\ref{HM t>1}.  When $k > 2t+1$, Theorem~\ref{HM t>1} gives two possible options for the largest family.  However, Frankl proved in~\cite{Frankl} that $\widetilde{\fis}_{t}$ is the largest family when $k > 2t+1$ provided that $n$ is sufficiently large.  Determining precisely how large $n$ must be would give a bound on $m$ for multiset systems above which $\widetilde{\fim}_{t}$ is the largest family. 

\section{Acknowledgements}

The first author's research was supported in part by NSERC Discovery Grant RGPIN-341214-2013.  The second author was supported by an NSERC doctoral scholarship.  We would like to thank Zoltan F{\"u}redi for bringing reference~\cite{MR3033661} to our attention.

\end{document}